\documentclass[a4paper,12pt]{amsart}

\usepackage{amssymb,amsmath,amsthm,mathrsfs,xspace, braket}
\usepackage[margin=1.0in]{geometry}
\numberwithin{equation}{section}
\usepackage[colorlinks=true]{hyperref}
\usepackage{aliascnt}
\usepackage{mathtools}
\mathtoolsset{showonlyrefs=true}
\usepackage{enumitem}
\setlist[enumerate]{itemsep=0pt,label=$({\rm\roman*})$,topsep=5pt}
\setlist[itemize]{itemsep=0pt,topsep=5pt,labelindent=\parindent,leftmargin=*}
\setlist[description]{itemsep=0pt,topsep=5pt,leftmargin=*}

\usepackage[all,2cell]{xy}
\objectmargin+{1mm}
\labelmargin+{0.8mm}
\SelectTips{cm}{12}
\newtheorem{thm}{Theorem}[section]

\newaliascnt{cor}{thm}

\aliascntresetthe{cor}

\newaliascnt{lem}{thm}
\newtheorem{lem}[lem]{Lemma}
\aliascntresetthe{lem}

\newaliascnt{prop}{thm}

\aliascntresetthe{prop}

\newaliascnt{conj}{thm}

\aliascntresetthe{conj}

\theoremstyle{definition}
\newaliascnt{dfn}{thm}

\aliascntresetthe{dfn}

\newaliascnt{rem}{thm}
\newtheorem{rem}[rem]{Remark}
\aliascntresetthe{rem}

\newcommand{\Char}{\operatorname{char}}
\newcommand{\Cor}{\operatorname{Cor}}
\newcommand{\Gm}{\mathbb{G}_{m}}
\DeclareMathOperator{\Jac}{Jac}
\newcommand{\Z}{\mathbb{Z}}
\newcommand{\et}{\mathrm{\acute et}}
\newcommand{\HM}{H_{\mathrm M}}
\newcommand{\Het}{H_{\mathrm{\acute et}}}
\newcommand{\cd}{\operatorname{cd}}

\title{Galois Symbols for a Jacobian and Multiplicative Groups}

\author[T. Hiranouchi]{Toshiro Hiranouchi}\address[T. Hiranouchi]{
Department of Basic Sciences, Graduate School of Engineering, 
Kyushu Institute of Technology, 
1-1 Sensui-cho, Tobata-ku, Kitakyushu-shi, 
Fukuoka, 804-8550 JAPAN}
\email{hira@mns.kyutech.ac.jp}

\author[R. Sugiyama]{Rin Sugiyama} \address[R. Sugiyama]{
Department of Mathematics, Physics and Computer Science, 
Japan Women's University, 2-8-1 Mejirodai, Bunkyo-ku, Tokyo, 112-8681 JAPAN
}\email{sugiyamar@fc.jwu.ac.jp}


\begin{document}
\date{July 13, 2026}

\begin{abstract}
Let $C$ be a smooth projective geometrically connected curve over a
field $k$ with a $k$-rational point. 
Let $J$ be the Jacobian variety of $C$.
For an integer $r\geq 1$ and a positive integer
$n$ prime to the characteristic of $k$, we prove that the Galois symbol map 
\[
 K(k;J,\Gm,\ldots,\Gm)/n
 \to
 \Het^{r+1}\bigl(k,J[n]\otimes\mu_n^{\otimes r}\bigr)
\]
is injective, where the multiplicative group $\Gm$ occurs $r$ times.  
The proof uses Akhtar's description of higher Chow groups of zero-cycles and the
Beilinson--Lichtenbaum theorem. The case $r=1$ recovers a theorem of
Spiess.  
\end{abstract}
\maketitle

\section{Introduction}

Let $G_1,\ldots,G_s$ be semi-abelian varieties over a field $k$.
Somekawa introduced the $K$-group
$K(k;G_1,\ldots,G_s)$, now called the Somekawa $K$-group,
by generators over finite extensions of $k$ and by projection-formula
and reciprocity relations \cite{Som90}.  If $n$ is prime to
the characteristic $\Char(k)$ of $k$, 
he also defined the Galois symbol map 
\begin{equation}\label{eq:gen}
 S_n(k;G_1,\ldots,G_s)\colon
 K(k;G_1,\ldots,G_s)/n
 \to
 \Het^s\bigl(k,G_1[n]\otimes\cdots\otimes G_s[n]\bigr);
\end{equation}
see \cite[Lemma 1.6]{Som90}. Somekawa conjectured that this map is
injective \cite[Remark 1.7]{Som90}.  The conjecture is not true for
all semi-abelian varieties; see \cite{SY09}.  It is therefore useful
to find natural classes for which the symbol is injective.

Let $C$ be a smooth projective geometrically connected curve over $k$
with a $k$-rational point, and let $J=\Jac_C$ be its Jacobian.  Spiess
proved that
\[
S_n(k;\Gm,J)\colon K(k;\Gm,J)/n
 \to \Het^2(k,\mu_n\otimes J[n])
\]
is injective (\cite[Appendix, Theorem~6.1]{Yam05}).  His proof uses
the Bloch--Ogus theory and the Merkurjev--Suslin theorem.  The purpose of this note is to extend this result to an arbitrary
number of copies of $\Gm$.

\begin{thm}\label{thm:main}
Let $k$ be a field, and let $C$ be a projective smooth geometrically connected curve over $k$ with a $k$-rational point.  Let $J$ be the
Jacobian variety of $C$.  Let $r\geq 1$, and let $n$ be a positive
integer prime to the characteristic $\Char(k)$ of $k$.  Then the Galois symbol map
\begin{equation}\label{eq:main}
 S_n^{J,r}\colon
 K(k;J,\underbrace{\Gm,\ldots,\Gm}_{r})/n
 \to
 \Het^{r+1}\bigl(k,J[n]\otimes\mu_n^{\otimes r}\bigr)
\end{equation}
is injective.
\end{thm}

The permutation of factors sends
$\set{z,a_1,\ldots,a_r}_{E/k}$ to
$\set{a_1,\ldots,a_r,z}_{E/k}$.  After the corresponding permutation
of the coefficient modules, the two Galois symbols differ by the sign
$(-1)^r$.  Hence the theorem is equivalent to the injectivity of
\[
 K(k;\underbrace{\Gm,\ldots,\Gm}_{r},J)/n
 \to 
 \Het^{r+1}\bigl(k,\mu_n^{\otimes r}\otimes J[n]\bigr).
\]
For $r=1$, after permuting the two factors,
\autoref{thm:main} gives another proof of Spiess's theorem \cite[Appendix, Theorem 6.1]{Yam05}.  
The proof reduces to two known results.  First, Akhtar proved that
a mixed $K$-group attached to $CH_0(C)$ and $r$ copies of
$\Gm$ is naturally isomorphic to the higher Chow group $CH^{r+1}(C,r)$ 
(\cite[Theorem~6.1]{Akh04}).  A rational point gives a decomposition of
this mixed $K$-group into
\[
 K_r^M(k)\oplus K(k;J,\Gm,\ldots,\Gm).
\]
Second, the Beilinson--Lichtenbaum comparison theorem implies that the
finite coefficient \'etale cycle map in bidegree $(r+2,r+1)$ is
injective (\cite[Corollary~1.2]{GL01}, \cite[Theorem~6.17]{Voe11}). 

\subsection*{Notation} 
Throughout this paper, a curve over a field means an integral
separated scheme of finite type over that field of dimension one.

\subsection*{Acknowledgements} 
The first author was supported by JSPS KAKENHI Grant Number 24K06672.
The second author was supported by JSPS KAKENHI Grant Number 21K03188.

\section{Preliminaries}

Let $C$ be a smooth projective geometrically connected curve over $k$ with $P\in C(k)$, and $J = \mathrm{Jac}_C$ be the Jacobian of $C$.

\subsection{Mixed \texorpdfstring{$K$}{K}-groups and higher Chow groups}
For every field extension $E/k$, we write $P_E$ for the base change of $P$.
In \cite{Akh04}, Akhtar defined the mixed $K$-group
\[
 K_r\bigl(k;CH_0(C),\Gm\bigr)
 :=K\bigl(k;CH_0(C),\underbrace{\Gm,\ldots,\Gm}_{r}\bigr).
\]
It is generated by symbols
$\set{z,a_1,\ldots,a_r}_{E/k}$, where $E/k$ is finite,
$z\in CH_0(C_E)$, and $a_i\in E^\times$.  The defining relations are
the projection-formula and reciprocity relations; see the definition of
the mixed $K$-group in \cite[Section~3]{Akh04}. 
We also define 
\[
K_r(k;J,\Gm) :=  K(k;J,\underbrace{\Gm,\ldots,\Gm}_{r}).
\]
Let $n$ be an integer prime to $\Char(k)$.  For a finite extension $E/k$, let
\[
 \delta_E\colon E^\times/(E^\times)^n\to \Het^1(E,\mu_n)
\]
and
\[
 \delta_{J,E}\colon J(E)/nJ(E)\to \Het^1(E,J[n])
\]
be the Kummer connecting homomorphisms.  The Galois symbol map
\[
 S_n^{J,r}\colon K_r(k;J,\Gm)/n
 \to \Het^{r+1}\bigl(k,J[n]\otimes\mu_n^{\otimes r}\bigr)
\]
is induced on symbols by
\begin{equation}\label{eq:GS}
 S_n^{J,r}\bigl(\set{z,a_1,\ldots,a_r}_{E/k}\bigr)
 =\Cor_{E/k}\bigl(
 \delta_{J,E}(z)\cup\delta_E(a_1)\cup\cdots\cup\delta_E(a_r)
 \bigr);
\end{equation}
see \cite[Lemma~1.6]{Som90}.
The point $P$ gives, for every finite extension $E/k$, the decomposition
\[
 CH_0(C_E)=\Z[P_E]\oplus CH_0(C_E)^0
          =\Z[P_E]\oplus J(E).
\]
By \cite[Corollary~5.4.5]{Akh00} and
\cite[Proposition~1.5]{Som90}, it induces a natural decomposition
\begin{equation}\label{eq:split}
 K_r\bigl(k;CH_0(C),\Gm\bigr)
 \simeq
 K_r^M(k)\oplus
 K_r(k;J,\Gm).
\end{equation}
We use cohomological notation for higher Chow groups.  By the
Nesterenko--Suslin--Totaro theorem, for every field $E$ there is a
natural isomorphism $K_r^M(E)\simeq CH^r(E,r)$ (\cite{NS90,Tot92}).  
We use this isomorphism to identify these two
groups.  Thus the notation $\set{a_1,\ldots,a_r}$ denotes both the
Milnor symbol and the corresponding class in $CH^r(E,r)$.  Since $C$
has dimension one, the group of zero-cycles in degree $r$ is
$CH^{r+1}(C,r)$.

\begin{thm}[{\cite[Theorem~6.1]{Akh04}}]\label{thm:Akh}
For every $r\geq0$, there is a natural isomorphism
\begin{equation}\label{eq:Akh}
 \varphi_r\colon
 CH^{r+1}(C,r)
 \overset{\sim}{\to}
 K_r\bigl(k;CH_0(C),\Gm\bigr).
\end{equation}
\end{thm}

We describe the inverse map of $\varphi_r$ in \eqref{eq:Akh}.
Let $E/k$ be finite and let $p_E\colon C_E\to C$ be the projection.
Via the canonical identification
\[
 E^\times\simeq CH^1(E,1),
\]
we regard each $a_i\in E^\times$ as its pull-back to
$CH^1(C_E,1)$ along the structure morphism of $C_E$.  The inverse map
$\varphi_r^{-1}$ is characterized by
\begin{equation}\label{eq:Akh-sym}
 \varphi_r^{-1}\bigl(\set{z,a_1,\ldots,a_r}_{E/k}\bigr)
 = (p_E)_*\bigl(z\cap a_1\cap\cdots\cap a_r\bigr)
\end{equation}
for $z\in CH_0(C_E)=CH^1(C_E)$ and $a_i\in E^\times$.  Here $\cap$
denotes the intersection product on higher Chow groups
\[
 \cap\colon CH^p(C_E,m)\times CH^q(C_E,n)
 \to CH^{p+q}(C_E,m+n),
\]
and
\[
 (p_E)_*\colon CH^{r+1}(C_E,r)\to CH^{r+1}(C,r)
\]
is the proper push-forward.  Thus
$z\cap a_1\cap\cdots\cap a_r\in CH^{r+1}(C_E,r)$.

Combining \eqref{eq:split} and \autoref{thm:Akh}, we obtain
\begin{equation}\label{eq:CH}
 CH^{r+1}(C,r)
 \simeq
 K_r^M(k)\oplus
 K_r(k;J,\Gm).
\end{equation}
Let $\iota_{J,r}$ denote the inclusion of the second factor in
\eqref{eq:CH}.  By \eqref{eq:Akh-sym},
\begin{equation}\label{eq:iota}
 \iota_{J,r}\bigl(\set{z,a_1,\ldots,a_r}_{E/k}\bigr)
 = (p_E)_*\bigl(z\cap a_1\cap\cdots\cap a_r\bigr)
\end{equation}
for $z\in J(E)=CH_0(C_E)^0$.
We will see the decomposition \eqref{eq:CH} via a decomposition of the Chow motive of $C$ (\autoref{lem:decomp_CH}).

\subsection{Chow motives and the action of correspondences}
We briefly recall effective Chow motives.
An effective Chow motive is a pair $M=(X,\pi)$ where $X$ is a smooth projective equidimensional variety over $k$ of $\dim X=d$ and $\pi$ is a projector in the ring $CH^{d}(X\times X)$ of correspondences on $X$.
We write $h(X)$ for the Chow motive $(X,1_X)$ where $1_X$ is the graph of the identity $\mathrm{id}_X$.

A correspondence $\pi\in CH^d(X\times X)$ induces an endomorphism $\pi_\ast$ of $CH^i(X,j)$ as follows.
For any $z\in CH^i(X,j)$, 
$$
\pi_\ast(z):={p_2}_\ast(\pi \cap p_1^\ast(z))\in CH^i(X,j)
$$
where $p_i:X\times X\to X$ is the $i$-th projection.
Note that ${1_X}_\ast$ is the identity.
For an effective Chow motive $M=(X,\pi)$, we define 
$
CH^i(M,j):=\pi_\ast CH^i(X,j)$.
Then we have a decomposition of $CH^i(X,j)$;
$$
CH^i(X,j)=CH^i(M,j)\oplus CH^i(M',j)
$$
where $M'=(X,1_X-\pi)$.

A $k$-rational point $P\in X(k)$ gives the following decomposition of $h(X)$;
$$
h(X)=h^0(X) \oplus h^{[1,2d-1]}(X) \oplus h^{2d}(X) 
$$
where 
\begin{align*}
&h^0(X)=(X, [P\times X]),\\
&h^{[1,2d-1]}(X)=(X, 1_X-[P\times X]-[X\times P])\\
&h^{2d}(X)=(X, [X\times P]).
\end{align*}
There are isomorphisms of Chow motives
$$
h^0(X)\simeq (\operatorname{Spec}(k), 1_{\operatorname{Spec}(k)}),\ h^{2d}(X)\simeq \mathbb{L}^{d}
$$
where $\mathbb{L}$ is the Lefschetz motive.
If $\dim X=1$, we write $h^1(X)$ for $h^{[1,2d-1]}(X)$, and in this case we have the following decomposition of $CH^i(X,j)$;
\begin{align}
    CH^i(X,j)=CH^i(h^0(X),j)\oplus CH^i(h^1(X),j)\oplus CH^i(h^2(X),j).
\end{align}

\begin{lem}\label{lem:decomp_CH}
	For any $r\ge 0$, the isomorphism \eqref{eq:Akh} induces isomorphisms 
    \begin{align}
    &CH^{r+1}(h^0(C),r)=0,\\
    &CH^{r+1}(h^1(C),r)\simeq K_r(k;J,\Gm),\ \text{and}\\
    &CH^{r+1}(h^2(C),r)\simeq K_r^M(k).
    \end{align}
\end{lem}
\begin{proof}
    Put $\alpha:=[P\times C], {}^t\alpha:=[C\times P]$.
    We first compute $\alpha_\ast, {}^t\alpha_\ast$ on $CH^1(C)$.
    For $z\in CH^1(C)$, we have
    $$
    \alpha\cap (z\times C)=0,\ {}^t\alpha \cap (z\times C)=z\times P
    $$
    which shows
    $$
    \alpha_\ast(z)=0,\ {}^t\alpha_\ast(z)=\deg(z)[P]\in \Z[P].
    $$
    This shows that the kernel of ${}^t\alpha_\ast$ on $CH^1(C)$ is equal to the kernel of the degree map
    $$
    \deg:CH^1(C) \to \Z
    $$ and proves the assertion in the case $r=0$.

    By \eqref{eq:Akh} and \eqref{eq:Akh-sym}, we take an element $w$ of $CH^{r+1}(C,r)$ of the following form
    $$
    w:=(p_E)_\ast\bigl(z\cap a_1\cap\cdots\cap a_r\bigr).
    $$
    By the base change formula for algebraic cycles with respect to proper push-forwards and flat pull-backs, it suffices to compute $\alpha_\ast(w)$ in case $E=k$.
    Then, we have
    \begin{align}
        \alpha_\ast(w)
        &={p_2}_\ast\bigl(\alpha\cap p_1^\ast(z\cap a_1\cap\cdots\cap a_r)\bigr)\\
        &={p_2}_\ast\bigl((\alpha\cap p_1^\ast z)\cap p_1^\ast(a_1\cap\cdots\cap a_r)\bigr)=0,
    \end{align}
    and
    \begin{align}
        {}^t\alpha_\ast(w)
        &={p_2}_\ast\bigl({}^t\alpha\cap p_1^\ast(z\cap a_1\cap\cdots\cap a_r)\bigr)\\
        &={p_2}_\ast\bigl(({}^t\alpha\cap p_1^\ast z)\cap p_1^\ast(a_1\cap\cdots\cap a_r)\bigr)\\
        &={p_2}_\ast\bigl([z\times P]\cap p_1^\ast(a_1\cap\cdots\cap a_r)\bigr)\\
        &=\deg(z)\bigl([P]\cap a_1\cap\cdots\cap a_r\bigr)\\
        &=\deg(z)\{a_1, \dots, a_r\}\in K_r^M(k).
    \end{align}
    Here in the last equality, we identify $CH^r(P,r)$ with $K^M_r(k)$.
    This shows that the kernel of ${}^t\alpha_\ast$ on $CH^{r+1}(C,r)$ is equal to the kernel of the map
    $$
    CH^{r+1}(C,r)\to K_r^M(k).
    $$
    Thus, the isomorphism \eqref{eq:Akh} induces
    $$
    CH^{r+1}(h^0(C),r)=0,\ CH^{r+1}(h^1(C),r)\simeq K_r(k;J,\Gm),\ CH^{r+1}(h^2(C),r)\simeq K_r^M(k).
    $$
\end{proof}
By \autoref{lem:decomp_CH}, the inclusion $\iota_{J,r}$ in \eqref{eq:iota} coincides with 
\begin{align}\label{eq:iota2}
    \iota_{J,r}: K_r(k;J,\Gm)\simeq CH^{r+1}(h^1(C),r)\hookrightarrow CH^{r+1}(C,r).
\end{align}

\subsection{\'Etale cycle map}
For a smooth scheme $X$, we use the identification
\[
 CH^q(X,2q-p)=\HM^p(X,\Z(q));
\]
see \cite{Blo86} and \cite[Theorem~1]{Voe02}.  Let
\begin{equation}\label{eq:rho}
 \rho_{X,n}^{i,j}\colon
 CH^{i}(X,j)/n
 \to
 \Het^{2i-j}\bigl(X,\mu_n^{\otimes i}\bigr)
\end{equation}
be the finite coefficient \'etale cycle map.  For $i=j+1$, its
injectivity follows from \cite[Corollary~1.2]{GL01} (see also \cite[Theorem~6.17]{Voe11}).  
Under the Nesterenko--Suslin--Totaro isomorphism
$CH^r(E,r)\simeq K_r^M(E)$, the finite coefficient \'etale cycle map is the
norm residue isomorphism
\[
 K_r^M(E)/n\overset{\sim}{\to}H^r(E,\mu_n^{\otimes r}),
 \qquad
 \set{a_1,\ldots,a_r}\mapsto
 \delta_E(a_1)\cup\cdots\cup\delta_E(a_r).
\]
Moreover, under the natural isomorphisms $CH^1(X)\simeq \Het^1(X, \Gm)$ and $CH^1(X,1)\simeq \Gamma(X,\Gm)$, the cycle maps $\rho^{1,0}_{X,n}, \rho^{1,1}_{X,n}$ coincide with the Kummer connecting homomorphisms
\begin{align}\label{eq:1-cyc}
    &\Het^1(X, \Gm)/n\to \Het^2(X,\mu_n),\\
    &\Gamma(X,\Gm)/n\to \Het^1(X,\mu_n).
\end{align}

\begin{lem}\label{lem:cycl}
    Let $X$ be a smooth projective variety over $k$ of dimension $d$.
    For any $\alpha\in CH^d(X\times X)$, put $\rho(\alpha):=\rho_{X\times X,n}^{d,0}(\alpha)$ in $\Het^{2d}(X\times X,\mu_n^{\otimes d})$. Then the following diagram commutes:
    $$
    \xymatrix{
    CH^i(X,j)/n \ar[r]^-{\rho^{i,j}_{X,n}} \ar[d]^{\alpha_*} & H_\et^{2i-j}(X,\mu_n^{\otimes i})\ar[d]^{\rho(\alpha)_*}\\
    CH^i(X,j)/n \ar[r]^-{\rho^{i,j}_{X,n}}  & H_\et^{2i-j}(X,\mu_n^{\otimes i})
    }
    $$
\end{lem}
\begin{proof}
    Let $p_i: X\times X\to X$ be the $i$-th projection.
    Take $z\in CH^i(X,j)$.
    By the compatibility of the cycle map with proper push-forward and flat pull-back, we have
\begin{align*}
 \rho^{i,j}_{X,n}\bigl(\alpha_\ast(z)\bigr)
 &=
 \rho^{i,j}_{X,n}
 \bigl({p_2}_\ast(\alpha\cap p_1^\ast z)\bigr)\\
 &=
 {p_2}_\ast
 \left(
 \rho^{d+i,j}_{X\times X,n}
 \bigl(\alpha\cap p_1^\ast z\bigr)
 \right)\\
 &=
 {p_2}_\ast
 \left(
 \rho(\alpha)\cup
 \rho^{i,j}_{X\times X,n}(p_1^\ast z)
 \right)\\
 &=
 {p_2}_\ast
 \left(
 \rho(\alpha)\cup
 p_1^\ast\bigl(\rho^{i,j}_{X,n}(z)\bigr)
 \right)\\
 &=
 \rho(\alpha)_\ast
 \bigl(\rho^{i,j}_{X,n}(z)\bigr).
\end{align*}
\end{proof}

For an effective Chow motive $M=(X,\pi)$, we define $H_\et^a(M,\mu_n^{\otimes b}):=\rho(\pi)_\ast H_\et^a(X,\mu_n^{\otimes b})$.
By \autoref{lem:cycl}, we have a \'etale cycle map for $M$;
$$
\rho^{i,j}_{M,n}:CH^i(M,j)/n\to H_\et^{2i-j}(M,\mu_n^{\otimes i}).
$$

\begin{lem}[\text{\cite[Lemma 2.5]{Yam05}}]\label{lem:decomp_Het}
    For $t\ge 1$, a $k$-rational point $P$ in $C$ gives
    \begin{align}
    &H^{s}_\et(h^0(C),\mu_n^{\otimes t})
    \simeq \Het^{s}(k,\mu_n^{\otimes t}),\\
    &H^s_\et(h^1(C),\mu_n^{\otimes t})
    \simeq H_\et^{s-1}(k,J[n]\otimes\mu_n^{\otimes (t-1)}),\\
    &H^{s}_\et(h^2(C),\mu_n^{\otimes t})
    \simeq H_\et^{s-2}(k,\mu_n^{\otimes (t-1)}).
    \end{align}
\end{lem}
\begin{proof}
    Put
    $\alpha:=[P\times C]$,
    ${}^t\alpha:=[C\times P]$, and 
    $\pi:=1_C-\alpha-{}^t\alpha$.
    Thus
    \[
    h^0(C)=(C,\alpha),\qquad
    h^1(C)=(C,\pi),\qquad
    h^2(C)=(C,{}^t\alpha).
    \]
    Since $
    h^0(C)\simeq (\operatorname{Spec}(k), 1_k), h^2(C)\simeq \mathbb{L},
    $
    the statement is clear for $h^0(C)$ and $h^2(C)$.
    Indeed, for $\rho(\alpha)=\rho^{1,0}_{C\times C,n}([P\times C])$, we have
    $$
    \rho(\alpha)_\ast: \Het^s(C,\mu_n^{\otimes t})\xrightarrow{i^\ast} \Het^s(k,\mu_n^{\otimes t})\xrightarrow{f^\ast} \Het^s(C,\mu_n^{\otimes t})
    $$
    and
    $$
    \rho({}^t\alpha)_\ast: \Het^s(C,\mu_n^{\otimes t})\xrightarrow{f_\ast} \Het^{s-2}(k,\mu_n^{\otimes (t-1)})\xrightarrow{i_\ast} \Het^s(C,\mu_n^{\otimes t}),
    $$
    where $i:P\hookrightarrow C$ is the closed immersion and $f:C\to \operatorname{Spec}(k)$ is the structure map.
    
    For $h^1(C)$, we use the Hochschild--Serre spectral sequence
    \[
    E_2^{a,b}
    =\Het^a\bigl(k, \Het^b(\overline C,
     \mu_n^{\otimes t})\bigr)
 \Longrightarrow
 \Het^{a+b}\bigl(C,\mu_n^{\otimes t}\bigr).
\]
See \cite[Theorem~12.7 and Example~12.8]{Mil13} and
\cite[Chapter~II, Section~4, Theorem~2.4.1]{NSW08}.  This spectral
sequence is multiplicative; on its $E_2$-page the product is induced
by the cup products in continuous Galois cohomology and geometric
\'etale cohomology.  See \cite[pp.~118--119]{HS53} and
\cite[Chapter~I, Section~4]{NSW08}.

The Kummer sequence and the trace isomorphism give
\[
 \Het^0(\overline C, \mu_n^{\otimes t}) \simeq \mu_n^{\otimes t},\ \Het^1(\overline C,\mu_n^{\otimes t})
 \simeq J[n]\otimes\mu_n^{\otimes (t-1)},\ \mbox{and} 
 \
 \Het^2(\overline C,\mu_n^{\otimes t})
 \simeq\mu_n^{\otimes (t-1)};
\]
see \cite[Proposition~14.2 and Theorem~24.1]{Mil13} and
\cite[Sections~8.2--8.3]{Fu15}. 

    The actions of $\alpha$ and ${}^t\alpha$ are compatible with the
    Hochschild--Serre spectral sequence. On the $E_2$-page, their
    actions are induced by the actions of the corresponding
    correspondences on $\Het^b(\overline C,\mu_n^{\otimes t})$.
    After base change to $\overline k$, the same formulas as above show
    that $\rho(\alpha)_\ast$ is the identity on
    $\Het^0(\overline C,\mu_n^{\otimes t})$ 
    and is zero on
    $\Het^1(\overline C,\mu_n^{\otimes t})$ 
    and $\Het^2(\overline C,\mu_n^{\otimes t})$.
    Similarly, $\rho({}^t\alpha)_\ast$ is the identity on
    $\Het^2(\overline C,\mu_n^{\otimes t})$
    and is zero on
    $\Het^0(\overline C,\mu_n^{\otimes t})$
    and $\Het^1(\overline C,\mu_n^{\otimes t})$.
    Since
    \[
    \rho(\pi)_\ast
    =
    1-\rho(\alpha)_\ast-\rho({}^t\alpha)_\ast,
    \]
    it follows that $\rho(\pi)_\ast$ acts as zero on
    $\Het^0(\overline C,\mu_n^{\otimes t})$ 
    and $\Het^2(\overline C,\mu_n^{\otimes t})$,
    and as the identity on
    $\Het^1(\overline C,\mu_n^{\otimes t})$.
    Therefore the Hochschild--Serre spectral sequence obtained by
    applying $\rho(\pi)_\ast$ has only the row $b=1$, namely
    \[
    E_2^{a,1}
    =
    \Het^a\bigl(
    k,J[n]\otimes\mu_n^{\otimes(t-1)}
    \bigr).
    \]
    Hence there are no nonzero differentials, and we obtain
    \[
    H^s_\et(h^1(C),\mu_n^{\otimes t})
    =
    \rho(\pi)_\ast
    \Het^s(C,\mu_n^{\otimes t})
    \simeq
    \Het^{s-1}\bigl(
    k,J[n]\otimes\mu_n^{\otimes(t-1)}
    \bigr).
    \]
This shows the assertion for $h^1(C)$.
\end{proof}

For $r\ge 0$, let $\iota^\et_{J,r}$ denote the inclusion 
\begin{align}\label{eq:iota_et}
 \iota^\et_{J,r}\colon
 \Het^{r+1}\bigl(k,J[n]\otimes\mu_n^{\otimes r}\bigr)
 \hookrightarrow
 \Het^{r+2}\bigl(C,\mu_n^{\otimes(r+1)}\bigr)
\end{align}
induced by the Hochschild--Serre spectral sequence and the
Chow--K\"unneth splitting, with the total-complex convention of
\cite[Chapter~II, Section~2, Definition~2.2.3 and p.~104]{NSW08}.

\begin{lem}\label{lem:KJ}
For every $z\in J(k)=CH_0(C)^0$, one has
\begin{equation}\label{eq:KJ}
 \rho_{C,n}^{1,0}(z)
 =\iota_{J,0}^{\et}\bigl(\delta_J(z)\bigr)
 \quad\text{in }H_\et^2(C,\mu_n).
\end{equation}
Equivalently, under the identification
\[
 H_\et^2\bigl(h^1(C),\mu_n\bigr)\simeq \Het^1(k,J[n]),
\]
the cycle map
\[
 \rho_{h^1(C),n}^{1,0}\colon J(k)/n
 \to \Het^1(k,J[n])
\]
coincides with the Kummer connecting homomorphism $\delta_J$.
\end{lem}

\begin{proof}
Under the identification $CH^1(C)=\operatorname{Pic}(C)$, the map
$\rho_{C,n}^{1,0}$ is the connecting homomorphism associated with the
Kummer sequence on $C$. Since a degree-zero class belongs to the
$h^1(C)$-part, the compatibility of the Kummer sequence with the
Hochschild--Serre edge morphism gives the commutative diagram
\[
\xymatrix{
 J(k)/n\ar[r]^-{\rho_{h^1(C),n}^{1,0}}\ar[d]_{\delta_J}
 &H_\et^2(h^1(C),\mu_n)\ar[d]^{\simeq}\\
 \Het^1(k,J[n])\ar@{=}[r]&\Het^1(k,J[n]).
}
\]
This proves the equivalent assertion, and \eqref{eq:KJ} follows by
including the $h^1(C)$-part into $H_\et^2(C,\mu_n)$.
\end{proof}

\begin{lem}\label{lem:HS-sign}
Let $f\colon C\to\operatorname{Spec}(k)$ be the structure morphism.
For $\alpha\in \Het^1(k,J[n])$ and $\beta\in \Het^r(k,\mu_n^{\otimes r})$, 
one has
\begin{equation}\label{eq:HS-sign}
 \iota_{J,0}^{\et}(\alpha)\cup f^*\beta
 =(-1)^r\iota_{J,r}^{\et}(\alpha\cup\beta).
\end{equation}
\end{lem}
\begin{proof}
In the standard total complex for the Hochschild--Serre spectral
sequence, the product of classes of bidegrees $(a,b)$ and $(a',b')$
acquires the sign $(-1)^{ba'}$; see
\cite[Chapter II, Section~1]{HS53} and
\cite[Chapter~II, Section~4, proof of Theorem~2.4.1]{NSW08}.
Here $\alpha$ has bidegree $(1,1)$ and $\beta$ has bidegree $(r,0)$.
Thus their product acquires the sign
$(-1)^{1\cdot r}=(-1)^r$.
Under the Chow--K\"unneth splitting, this is precisely
\eqref{eq:HS-sign}.
\end{proof}


\section{Proof of the main theorem}

\begin{proof}[Proof of \autoref{thm:main}]
By \autoref{lem:decomp_CH}, the map $\iota_{J,r}$ is the inclusion
of a direct summand. Hence it remains injective after reduction
modulo $n$. From \autoref{lem:decomp_CH} and
\autoref{lem:decomp_Het}, we therefore have a commutative diagram
\[
\xymatrix{
 K_r(k;J,\Gm)/n\ar@{^{(}->}[r]^{\iota_{J,r}}_{\eqref{eq:iota2}}\ar[d]^{\rho_{h^1(C),n}^{r+1,r}}
 &CH^{r+1}(C,r)/n\ar@{^{(}->}[d]^{\rho_{C,n}^{r+1,r}}\\
 \Het^{r+1}\bigl(k,J[n]\otimes\mu_n^{\otimes r}\bigr)
 \ar@{^{(}->}[r]^-{\iota_{J,r}^{\et}}_{\eqref{eq:iota_et}}
 &\Het^{r+2}\bigl(C,\mu_n^{\otimes(r+1)}\bigr).
}
\]
The right vertical map is injective by
\cite[Corollary~1.2]{GL01} (see also
\cite[Theorem~6.17(1)]{Voe11}); hence the left vertical map
\[
 \rho_{h^1(C),n}^{r+1,r}\colon
 K_r(k;J,\Gm)/n
 \longrightarrow \Het^{r+1}\bigl(k,J[n]\otimes\mu_n^{\otimes r}\bigr)
\]
is also injective.

It remains to compare the \'etale cycle map
$\rho_{h^1(C),n}^{r+1,r}$ with the Galois symbol map
$S_n^{J,r}$. We first consider symbols defined over the base field.
Let $f\colon C\to\operatorname{Spec}(k)$ be the structure morphism.
For $z\in J(k)$ and $a_1,\ldots,a_r\in k^\times$, the compatibility
of the cycle map with the intersection product and the cup product gives
\begin{align*}
 \iota_{J,r}^{\et}\circ\rho_{h^1(C),n}^{r+1,r}
 \bigl(\{z,a_1,\ldots,a_r\}_{k/k}\bigr)
 &=\rho_{C,n}^{r+1,r}
 \bigl(z\cap a_1\cap\cdots\cap a_r\bigr)\\
 &=\rho_{C,n}^{1,0}(z)\cup
   \rho_{C,n}^{1,1}(a_1)\cup\cdots\cup
   \rho_{C,n}^{1,1}(a_r)\\
 &\stackrel{\eqref{eq:KJ}}{=}
 \iota_{J,0}^{\et}\bigl(\delta_J(z)\bigr)\cup
 f^*\delta(a_1)\cup\cdots\cup f^*\delta(a_r)\\
 &\stackrel{\eqref{eq:HS-sign}}{=}
 (-1)^r\iota_{J,r}^{\et}\bigl(
 \delta_J(z)\cup\delta(a_1)\cup\cdots\cup\delta(a_r)
 \bigr).
\end{align*}
Thus the desired comparison holds for symbols over $k$.

Now let $E/k$ be a finite extension, and let
$p_E\colon C_E\to C$ be the projection. Write
$\iota_{J_E,r}^{\et}$ for the corresponding inclusion over $E$.
The naturality of the Hochschild--Serre spectral sequence and the
compatibility of proper push-forward with trace give a commutative
diagram
\begin{equation}\label{eq:trace}
\vcenter{
\xymatrix{
 \Het^{r+1}\bigl(E,J[n]\otimes\mu_n^{\otimes r}\bigr)
 \ar[r]^-{\iota_{J_E,r}^{\et}}\ar[d]_{\Cor_{E/k}}
 &\Het^{r+2}\bigl(C_E,\mu_n^{\otimes(r+1)}\bigr)
 \ar[d]^{(p_E)_*}\\
 \Het^{r+1}\bigl(k,J[n]\otimes\mu_n^{\otimes r}\bigr)
 \ar[r]^-{\iota_{J,r}^{\et}}
 &\Het^{r+2}\bigl(C,\mu_n^{\otimes(r+1)}\bigr).
}}
\end{equation}
See \cite[Expos\'e~IX, Section~5]{SGA4} and
\cite[Section~8.2]{Fu15}.
Applying the comparison over $E$ and using the compatibility of the
cycle map with proper push-forward, we obtain, for
$z\in J(E)$ and $a_1,\ldots,a_r\in E^\times$,
\begin{align*}
 &\iota_{J,r}^{\et}\circ\rho_{h^1(C),n}^{r+1,r}
 \bigl(\{z,a_1,\ldots,a_r\}_{E/k}\bigr)\\
 &\quad=
 \rho_{C,n}^{r+1,r}\Bigl(
 (p_E)_*\bigl(z\cap a_1\cap\cdots\cap a_r\bigr)
 \Bigr)\\
 &\quad=
 (p_E)_*\rho_{C_E,n}^{r+1,r}
 \bigl(z\cap a_1\cap\cdots\cap a_r\bigr)\\
 &\quad=
 (-1)^r(p_E)_*\iota_{J_E,r}^{\et}\bigl(
 \delta_{J,E}(z)\cup\delta_E(a_1)\cup\cdots\cup\delta_E(a_r)
 \bigr)\\
 &\quad\stackrel{\eqref{eq:trace}}{=}
 (-1)^r\iota_{J,r}^{\et}\Cor_{E/k}\bigl(
 \delta_{J,E}(z)\cup\delta_E(a_1)\cup\cdots\cup\delta_E(a_r)
 \bigr)\\
 &\quad\stackrel{\eqref{eq:GS}}{=}
 (-1)^r\iota_{J,r}^{\et}\circ S_n^{J,r}
 \bigl(\{z,a_1,\ldots,a_r\}_{E/k}\bigr).
\end{align*}
Since such symbols generate $K_r(k;J,\Gm)$ and
$\iota_{J,r}^{\et}$ is injective, we obtain
\[
 \rho_{h^1(C),n}^{r+1,r}=(-1)^rS_n^{J,r}.
\]
Since $\rho_{h^1(C),n}^{r+1,r}$ is injective, the Galois symbol map
$S_n^{J,r}$ is injective.
\end{proof}

\begin{rem}[\text{\cite[Lemma~3.2]{HS25}, \cite{Hir24}}]\label{prop:div}
We record a stable-range $l$-divisibility result for Somekawa $K$-groups attached to several semi-abelian varieties and copies of $\Gm$.
Let $G_1,\ldots, G_q\ (q\ge 1)$ be semi-abelian varieties over $k$
, and let
$l\neq\Char(k)$ be a prime.  Assume that
$s=\cd_l(k)<\infty$.  If $r\ge s$, then
$K(k;G_1,\ldots,G_q,
 \underbrace{\Gm,\ldots,\Gm}_{r})$
is $l$-divisible.  Consequently, its quotient modulo $l^m$ is zero
for every $m\geq1$, and the corresponding Galois symbol map modulo
$l^m$ is injective.
\end{rem}


\end{document}